%% file: new_version.tex
\documentclass[12pt,letterpaper]{amsart}
\usepackage[dvipsnames]{xcolor}
\usepackage{amsmath,amsthm,amssymb}
\usepackage[margin=1.25in]{geometry} 
\usepackage{comment}
\usepackage{dsfont}
\usepackage{ textcomp }
\usepackage{tkz-euclide}
\usepackage{fancyhdr}
\usepackage{calligra,mathrsfs}
\usepackage[shortlabels]{enumitem}
\usepackage{tikz-cd}
\usepackage{hyperref}
\hypersetup{colorlinks=true, linkcolor=blue, citecolor=black}
\usepackage{cleveref}
\newcommand{\ep}{\varepsilon}
\newcommand{\trdeg}{\textnormal{trdeg}}

\input{lettercoms}
\input{declarations}

\theoremstyle{plain} 
\newtheorem{theorem}{Theorem}[section]
\newtheorem{definition}[theorem]{Definition}
\newtheorem*{theorem*}{Theorem}
\newtheorem{prop}[theorem]{Proposition}

\newtheorem{lemma}[theorem]{Lemma}

\newtheorem{conj}[theorem]{Conjecture}
\newtheorem*{claim}{Claim}

\newtheoremstyle{remboldstyle}
  {}{}{}{}{\bfseries}{.}{.5em}{{\thmname{#1}}{\thmnumber{#2}}{\thmnote{ (#3)}}}
\theoremstyle{remboldstyle}
\newtheorem*{remark*}{Remark}
\newtheorem{example}[theorem]{Example\ }

\newcommand{\ul}{\underline}
\newcommand{\rk}{\textnormal{rk}}
\newcommand{\prk}{\textnormal{prk}}

\newcommand{\str}{\textnormal{str}}
\newcommand{\ark}{\textnormal{ark}}
\newcommand{\Brk}{\textnormal{Brk}}
\newcommand{\grk}{\textnormal{grk}}
\renewcommand{\char}{\textnormal{char}}

\usepackage[
backend=biber,
style=alphabetic,maxbibnames=9
]{biblatex}
\title{Strength is bounded linearly by Birch rank}
\date{}

\author{Benjamin Baily}
\address{Department of Mathematics, University of Michigan, Ann Arbor, MI}
\email{\href{mailto:bbaily@umich.edu}{bbaily@umich.edu}}
\thanks{BB was supported by NSF grant DMS-2101075}

\author{Amichai Lampert}
\address{Department of Mathematics, University of Michigan, Ann Arbor, MI}
\email{\href{mailto:amichai@umich.edu}{amichai@umich.edu}}
\thanks{AL was supported by NSF grant DMS-2402041}

\begin{document}

\begin{abstract}
    Let $f$ be a homogeneous polynomial over a field. For many fields, including number fields and function fields, we prove that the strength of $f$ is bounded above by a constant multiple of the Birch rank of $f.$ The constant depends only on the degree of $f$ and the absolute transcendence degree of the field. This is the first linear bound obtained for forms of degree greater than three, partially resolving a conjecture of Adiprasito, Kazhdan and Ziegler.

    Our result has applications for the Hardy-Littlewood circle method. The circle method yields an asymptotic formula for counting integral zeros of (collections of) homogeneous polynomials, provided the Birch rank is sufficiently large -- a natural \emph{geometric} condition. Our main theorem implies that these formulas hold even if we only assume a similar lower bound on the strength of the (collection of) homogeneous polynomials -- an \emph{arithmetic} condition which is a priori weaker. This answers questions of Cook-Magyar \cite{CM-primes} and Skinner \cite{Ski}, and also yields a new proof of a  seminal result of Schmidt \cite[Theorem II]{Sch} as a consequence of Birch's earlier work \cite{Bir}.

    Over finite fields we obtain a quasi-linear bound for partition rank of tensors in terms of analytic rank, improving Moshkovitz-Zhu's state of the art bound \cite{MZ-old}.    
\end{abstract}

\maketitle

\section{Introduction}

\subsection{Background} 
Let $K$ be a field and $f \in K[x_1,\ldots,x_n]$ a homogeneous polynomial of degree $d.$ In 1962, Birch  \cite{Bir} introduced the following notion of rank.

\begin{definition}[Birch rank]
    The \emph{Birch rank} of $f$ is 
    \[
    \Brk(f) = \codim_{\A^n} (x:\nabla f(x) = 0).
    \]
\end{definition}

For $K=\Q,$ Birch obtained an asymptotic formula for the number of integer zeroes of $f$ of bounded height, provided that $\Brk(f)$ is sufficiently large. We will call such a result a \emph{Birch-type result}. Birch-type results have since been obtained in many other settings, including number fields \cite{Ski}, function fields \cite{Lee}, collections of forms of different degrees \cite{BHB} and for solutions whose entries are all prime \cite{CM-primes}. In the context of the Hardy-Littlewood circle method, or more generally when estimating exponential sums involving $f$, large Birch rank is a natural hypothesis as it allows good control of these sums. 

In $1985,$ Schmidt \cite{Sch} defined another notion of rank for polynomials.

\begin{definition}[Strength]
    The \emph{strength} of $f$ is 
    \[
    \str(f) = \min\left\{s: f = \sum_{i=1}^s g_ih_i\right\},
    \]
    where $g_i,h_i\in K[x_1,\ldots,x_n]$ are forms of degree less than $d$.
\end{definition}

\begin{remark*}
    Strength is sometimes known as Schmidt rank in the literature. The term `strength' was introduced by Ananyan and Hochster \cite{AH}.
\end{remark*}
If $f = g_1h_1 + \dots + g_rh_r$, then $\{x: \forall i, g_i(x) = h_i(x) = 0\} \subseteq \{x: \nabla f(x) = 0\}$, so $\Brk(f)\leq 2\cdot\str(f)$ is an easy consequence of Krull's height theorem. The converse implication, however, is far from obvious and is the subject of this paper. Schmidt proved that when $K=\Q$ it is sufficient to assume only large strength in order to obtain an asymptotic formula for the number of integer solutions. We will call such a result a \emph{Schmidt-type result}. We emphasize that Schmidt did \emph{not} prove his result by showing that large strength implies large Birch rank, but instead by some ingenious alternative arguments. Because strength is defined via forms of lower degree, Schmidt's result lends itself naturally to inductive arguments with consequences for the number of integer zeros of \emph{arbitrary} forms.  

In the following decades, strength has found myriad applications ranging from additive combinatorics to commutative algebra. In 2009, motivated by the inverse conjecture for the Gowers norm, Green and Tao \cite{GT} proved that polynomials of high strength over finite fields have equidistributed values. Their result was extended and improved by many others, see e.g. \cite{KL,HS,TZ,Mil,Jan,BL,CM,LZ-rel,MZ}. In 2020, Ananyan and Hochster \cite{AH} proved Stillman's conjecture regarding the projective dimension of ideals in polynomial rings. Strength played a central role in their proof, as well as in subsequent related work, such as \cite{ESS1,Dra,BDE,ESS2}. 

In the original context of the Hardy-Littlewood circle method, however, Schmidt's result has not been extended to other settings. The methods used in \cite{Sch} have proven difficult to adapt elsewhere, and the analogous Schmidt-type results are explicitly left open in the work of Skinner \cite{Ski} and Cook-Magyar \cite{CM-primes}.   

\subsection{Results}

Our main result bridges the gap between Birch-type results and Schmidt-type results by showing that $\str(f) \ll_d \Brk(f)$ in all the extensions of Birch's theorem mentioned above. This converts all of the Birch-type results above into  Schmidt-type results, with the bound increasing by only a constant factor.  As stated in the abstract, this answers questions raised by Skinner and Cook-Magyar, and recovers Schmidt's seminal result as a consequence of Birch's theorem. See section 2 for details.

Bounds for strength in terms of Birch rank were obtained for cubics in \cite{Der}, \cite{AKZ}, and \cite{CM-cubic}; for quartics in \cite{KP}; and for arbitrary degree in \cite{LZ}, albeit only for some fields. More general results were obtained in \cite{BDLZ} and \cite{BDS}. Adiprasito, Kazhdan and Ziegler conjectured in \cite{AKZ} that linear bounds hold for all degrees.

\begin{conj}\label{conj:str-brk}
Suppose that $\char(K) = 0$  or $\char(K) > d.$  Then 
        \[
        \str(f) \ll_d \Brk(f).
        \]
\end{conj}

So far, this conjecture has been proven only for $d \le 3.$ Our main result partially resolves conjecture \ref{conj:str-brk}. To state it, we first define
 \[
    r = r(K) = \begin{cases}
        \trdeg(K/\Q)+1\ ,\ &\char(K) = 0, \\
        \max(\trdeg(K/\F_p),1), \ &\char(K) = p > 0.
    \end{cases}.
\]
Our result is then as follows.

\begin{theorem}\label{thm:str-brk}
    Suppose that $\char(K) = 0$  or $\char(K) > d.$ 
    \begin{enumerate}
        \item If $K$ is infinite then $\str(f) \le  C(d,K)\cdot \Brk(f),$ where 
        \[
        C(d,K) = (2^{d-1}-1)(d-1)^{r+1}\binom{d}{\lfloor d/2 \rfloor}.
        \]
        \item If $K$ is finite then $\str(f) \ll_d \Brk(f)\cdot  \log (\Brk(f)+1).$
    \end{enumerate}
\end{theorem}

Theorem \ref{thm:str-brk} for finite fields complements the results of Moshkovitz-Zhu \cite{MZ} to improve their current record bound for partition rank in terms of analytic rank for tensors. The relevant definitions are given in \S 2. We obtain the following.

\begin{theorem}\label{thm:ark-prk}
    Suppose $F$ is a  multi-linear map of degree $d$ over a finite field. Then 
    \[
    \prk(F) \ll_d \ark(F)\cdot \log(\ark(F)+1).
    \] 
\end{theorem}

\begin{remark*}
Theorems \ref{thm:str-brk} and \ref{thm:ark-prk} for finite fields were recently proved independently in \cite{MZ} based on recent results in  \cite{ChenYe,MZ-grar}. The results of \cite{ChenYe,MZ-grar} are obtained by a different method than ours and lead to slightly weaker bounds.
\end{remark*}

\subsection*{Layout of the paper} In \S 2 we state our results for multi-linear forms in detail and describe some applications of theorem \ref{thm:str-brk} for the Hardy-Littlewood circle method. In \S 3 we recall known results and use these to show that our results all follow from  theorem \ref{thm:ark-grk}. We then describe the structure of the proof of this theorem. In \S 4 we make some reductions regarding the fields we must consider, and establish results relating point counting to dimension - including an asymptotic characterization of geometric rank. In \S 5 we prove theorem \ref{thm:ark-grk} in positive characteristic and in \S 6 we prove theorem \ref{thm:ark-grk} in characteristic zero.

\subsection*{Acknowledgements} The first author is grateful to Karen Smith for feedback on \S 4. The second author is grateful to Andrew Snowden for suggesting the Weil restriction argument of lemma \ref{lemma:separable} and to Guy Moshkovitz for valuable conversations.

\section{Multi-linear results and applications of theorem \ref{thm:str-brk}}

\subsection{Multi-linear results}
Let $V$ be a finite-dimensional vector space over $K$ and let $F:V^d\to K$ be a multi-linear map, i.e. a  tensor of degree $d.$ The analogue of Birch rank in the multi-linear setting is \emph{geometric rank}, which was defined by  Kopparty, Moshkovitz and Zuiddam \cite{KMZ20}.

\begin{definition}[Geometric rank]
The \emph{geometric rank} of $F$ is $\grk(F) = \codim_{V^{d-1}} S_F,$ where 
\[
S_F = \left( (x_1,\ldots,x_{d-1})\in V^{d-1}:F(x_1,\ldots,x_{d-1},\cdot) \equiv 0 \right).
\]
\end{definition}

There is also a multi-linear analogue of strength, called \emph{partition rank}. This was introduced by Naslund in \cite{Nas}.

\begin{definition}[Partition rank]
    We say that $F$ has \emph{partition rank one} if there exists a subset $\emptyset\neq I\subsetneq [d]$ such that $F(x_1,\ldots,x_d) = G((x_i)_{i\in I})\cdot H((x_j)_{j\in [d]\setminus I})$, where $G:V^{|I|}\to K$ and $H:V^{d-|I|}\to K$ are nonzero multi-linear forms. In general, the \emph{partition rank} of $F$ is 
    \[
    \prk(F) = \min\{r: F = F_1+\ldots+F_r\},
    \]
    where the $F_i:V^d\to K$ are multi-linear forms of partition rank one.
\end{definition}

The multi-linear version of theorem \ref{thm:str-brk} can then be stated as follows.

\begin{theorem}\label{thm:prk-grk} Let $r = r(K)$ be as in theorem \ref{thm:str-brk}.
    \begin{enumerate}
    \item For infinite $K$ we have 
    \[
      \prk(F) \le (2^{d-1}-1)(d-1)^r  \grk(F),
    \]   
    \item For finite $K$ we have
    \[
     \prk(F) \ll_d \grk(F)\cdot  \log (\grk(F)+1).
    \]
    \end{enumerate}
\end{theorem}

\begin{remark*}
    Note that there is no restriction on the characteristic in the above theorem.
\end{remark*}

All of our results will follow from a novel technical result relating geometric rank and another notion -- \emph{analytic rank}. For finite fields, analytic rank is an arithmetic version of geometric rank, introduced by Gowers and Wolf in \cite{GWol}. We  extend their definition to infinite fields. 

\begin{definition}[Analytic rank]\mbox{} 
\begin{enumerate}
    \item For $K=\F_q,$ the \emph{analytic rank} of $F$ is 
    \[
    \ark(F) =  \log_q \left( \frac{q^{(d-1)\dim V}}{|S_F(K)|} \right).
    \]
    \item For infinite $K,$ the \emph{analytic rank} is 
    \[
    \ark(F) = \codim_{V^{d-1}} \overline{S_F(K)},
    \]
    where the closure is taken with respect to the Zariski topology.
\end{enumerate}
\end{definition} 

We now state our main technical result, which is of independent interest. Informally, we show that varieties defined by systems of multi-linear equations are full of
rational points.

\begin{theorem}\label{thm:ark-grk}\mbox{}
    \begin{enumerate}
        \item For infinite $K$ we have $\ark(F) \le (d-1)^r\cdot  \grk(F),$
        where $r = r(K)$ is defined as in theorem \ref{thm:str-brk}.
        \item For finite $K$ we have $\ark(F) \le (d-1)\cdot \grk(F).$
    \end{enumerate}
\end{theorem}

We will explain how this result implies theorems \ref{thm:str-brk}, \ref{thm:ark-prk} and \ref{thm:prk-grk} in \S 3. 

\subsection{Applications for the Hardy-Littlewood circle method}

Let $\underline{f} = (f_1,\ldots,f_s) $ be a collection of homogeneous forms of degree $d$ in $K[x_1,\ldots,x_n].$ All of the results we will discuss in this subsection apply also for systems of forms of varying degrees, but we stick to one fixed degree for clarity of exposition. We now need Birch rank for tuples of forms.

\begin{definition}[Birch rank for tuples]
    The \emph{Birch rank} of $\underline{f}$ is
    \[
    \Brk(\underline{f}) = \codim_{\A^n} (x:\rk(J(x)) < s),
    \]
    where $J(x)$ is the matrix of partial derivatives.
\end{definition}

When $K$ is a number field or a function field, we adopt Schmidt's terminology from \cite{Sch} and say that $\underline{f}$ is a \emph{Hardy-Littlewood system}, or HLS for short, when the number of common integer zeroes of $\underline{f}$ obeys the asymptotic formula predicted by the Hardy-Littlewood circle method. 
\begin{example}
    Let $K= \Q.$ If $\underline{f}$ is an HLS, then
    \[
    \{x\in [-P, P]^n\cap \Z^n: f_i(x) = 0\text{ for all }1\leq i\leq s\} = \mathfrak J\mathfrak S P^{n-ds} + o(P^{n-ds}),
    \]
    where $\mathfrak{J},\mathfrak{S}$ are the \textit{singular integral} and \textit{singular series} respectively. We note that $\mathfrak{J}\mathfrak{S} > 0$ whenever the vanishing locus of $\underline{f}$ contains a nonsingular rational point. For further details on the case $K = \Q$ see \cite[Section 3]{Sch}; for the case of a general number field see \cite[Lemma 6]{Ski}.
\end{example}
The following result was first proved by Birch for $K=\Q$ \cite{Bir} and then extended by Skinner to number fields \cite{Ski} and Lee to function fields \cite{Lee}. 

\begin{theorem}[Birch, Skinner, Lee]\label{thm:Birch}
    Let $K$ be a number field or a function field $\F_q(t)$ with characteristic $>d.$ If $\Brk(\ul{f}) > s(s+1)(d-1)2^{d-1}$ then $\ul{f}$ is a HLS.
\end{theorem}

Cook and Magyar \cite{CM-primes} proved the striking result that for $K=\Q$ and $d\ge 2$ a similar hypothesis implies an asymptotic for the number of integer solutions whose entries are all \emph{prime} numbers, with an appropriately modified asymptotic formula (For $d=1$ they require a different notion of rank, since for example the smooth linear form $x_1$ has no zeros with all entries prime, no matter how large $n$ is). When such a formula holds, we call $\ul{f}$ a \emph{prime HLS}.

\begin{theorem}[Cook-Magyar]\label{thm:prime-sols}
    Let $K=\Q$ and $d\ge 2.$ There exists a constant $\alpha(d,s)$ such that if $\Brk(\ul{f}) > \alpha(d,s)$ then $\ul{f}$ is a prime HLS.
\end{theorem}

We say that the above results are of \emph{Birch-type}, since the hypothesis is a lower bound on the Birch rank.

In a celebrated result \cite{Sch}, Schmidt extended Birch's original theorem to the setting where we assume only a lower bound on the strength of $\ul{f}.$ To state Schmidt's result, we first define strength for tuples of forms.

\begin{definition}[Strength for tuples]
    The \emph{strength} of $\ul{f}$ is 
    \[
    \str(\ul{f}) = \min\left\{\str(a_1f_1+\ldots+a_sf_s): a_1,\ldots,a_s\in K \textnormal{ not all zero}\right\},
    \]
    where the zero form has strength zero.
\end{definition}

Schmidt's result is as follows. 

\begin{theorem}[Schmidt]\label{thm:Schmidt}
    Let $K=\Q.$ There exists a constant $\beta = \beta(d)$ such that if $\str(\ul{f}) \ge \beta s^2$ then $\ul{f}$ is a HLS.
\end{theorem}

Recalling the easy inequality $\Brk(\ul{f}) \le 2\str(\ul{f}),$ we see that Schmidt's theorem implies Birch's theorem, albeit with a bound which has worse dependence on $d.$ The true utility of Schmidt's result, however, lies in the fact that forms not satisfying the hypotheses can by definition be decomposed into forms of lower degree. This inductive process allowed him to deduce results about asymptotic counts of integer zeroes for \emph{arbitrary} systems of forms. As we mentioned earlier, Schmidt did \emph{not} deduce his result from Birch's theorem by showing that $\str(f)$ is bounded above by some function of $\Brk(f),$ but instead by a novel alternative argument. We call theorem \ref{thm:Schmidt} a \emph{Schmidt-type} result, since the hypothesis is a lower bound on strength.

Schmidt's argument has proven difficult to adapt to other settings, and to our knowledge there are no further Schmidt-type results in the literature. Both Skinner \cite{Ski} and Cook-Magyar \cite{CM-primes} speculate about the possibility of obtaining Schmidt-type results in their settings and discuss the challenges this would present. 

Theorem \ref{thm:str-brk} directly implies that all of the Birch-type results above yield a Schmidt-type result with a comparable bound. Before stating this, we need an extension for tuples of forms.  

\begin{lemma}\label{lem:str-brk-tuple}
    When $K$ has characteristic zero we have
    \[
    \str(\ul{f}) \le Cs\cdot   (\Brk(\ul{f})+s-1),
    \]
    where $C=C(d,K)$ is the constant of theorem \ref{thm:str-brk}.
\end{lemma}

\begin{proof}
    This follows by an argument which we learned from A. Polishchuk. First, note that
    \[
    \{x\in \overline{K}^n: \rk (J(x)) < s\} = \bigcup_{\underline{a}\in \P_{\overline{K}}^{s-1}} \{x\in \overline{K}^n: \nabla (a_1f_1+\ldots+a_sf_s)(x) = 0\}.
    \]
    Consequently, we have
    \begin{equation}\label{eq:brk-tuple}
         \min_{\underline{a}\in \P_{\overline{K}}^{s-1}} \Brk(a_1f_1+\ldots+a_sf_s) \le \Brk(\ul{f}) + s-1.
    \end{equation}    
    Now, suppose the minimum is obtained for some $0\neq \ul{a}\in  \overline{K}^s$ and write $\eta = \Brk(\ul{a}\cdot \ul{f}).$ We claim that there exists $0\neq \ul{b}\in  K^s$ with $ \Brk(\ul{b}\cdot \ul{f}) \le s\eta.$ Theorem \ref{thm:str-brk} and inequality \eqref{eq:brk-tuple} then imply
    \[
    \str(\ul{b}\cdot \ul{f}) \le C\cdot s\eta \le Cs \cdot   (\Brk(\ul{f})+s-1),
    \]
    completing the proof. To prove the claim about the existence of such a $\ul{b},$ let $\mathcal{F}$ be the subspace of $\textnormal{sp}_{\overline{K}}(f_1,\ldots,f_s)$ spanned by the Galois conjugates of $\ul{a}\cdot \ul{f}.$ By Galois descent, this has a basis consisting of forms in $\textnormal{sp}_{K}(f_1,\ldots,f_s).$ Let $\ul{b}\cdot \ul{f}$ be one of these basis elements. Since $\dim (\mathcal{F}) \le s,$ we can write $\ul{b}\cdot \ul{f} = \sum_{i=1}^s \beta_i \sigma_i (\ul{a}\cdot \ul{f})$ where $\beta_i\in \overline{K}$ and the $\sigma_i$ are field automorphisms. Clearly we have $\Brk(\sigma_i (\ul{a}\cdot \ul{f})) = \Brk(\ul{a}\cdot \ul{f}) = \eta$ for all $1\le i\le s.$ Birch rank is subadditive, so we get $\Brk(\ul{b}\cdot \ul{f}) \le s\eta$ as claimed.
\end{proof}
    
    Now we can convert the above Birch-type results to Schmidt-type results. In particular, this yields a new proof of Schmidt's theorem \ref{thm:Schmidt}.
    
\begin{prop}
    There exists a constant $\gamma = \gamma(d)$ such that the following holds:
    \begin{enumerate}
        \item If $K$ is a number field and $\str(\ul{f}) > \gamma \cdot s^3$ then $\ul{f}$ is a HLS.
        \item Suppose that $K=\Q$ and $d\ge 2.$ If  $\str(\ul{f}) > \gamma \cdot \alpha(d,s)$ then $\ul{f}$ is a prime HLS, where $\alpha(d,s)$ is any constant such that theorem \ref{thm:prime-sols} holds. 
        \item Suppose that $K = \F_q(t)$ has characteristic $>d$ and that $s=1.$ If $\str(f) > \gamma$ then $f$ is a HLS. 
    \end{enumerate}
\end{prop} 

\begin{proof}
This follows immediately by combining theorems \ref{thm:Birch} and \ref{thm:prime-sols} with lemma \ref{lem:str-brk-tuple}.
\end{proof}

\begin{remark*}
    Two remarks are in order here. The first is that we get a result over function fields only for a single form. This is due to the absence of a suitable analogue of lemma \ref{lem:str-brk-tuple}. The second is that the above proposition reproves Schmidt's theorem \ref{thm:Schmidt} but with a cubic dependence on $s$ rather than quadratic. However, the quadratic dependence can also be recovered by using a refinement of Birch's theorem due to Dietmann \cite{Diet} and independently Schindler \cite{Schin}. We do not give a proof of this here, but the interested reader should have no difficulty filling in the details.
\end{remark*} 

\section{Reduction to theorem \ref{thm:ark-grk} and proof sketch}

In this section we assume theorem \ref{thm:ark-grk} holds and use it to deduce our other results: theorems \ref{thm:str-brk}, \ref{thm:ark-prk} and \ref{thm:prk-grk}. Then we discuss the plan of the proof of theorem \ref{thm:ark-grk}.

\subsection{Proof of theorem \ref{thm:str-brk}} We begin by explaining how theorem \ref{thm:str-brk} follows from its multi-linear analogue theorem \ref{thm:prk-grk}. Suppose that $V$ is a finite-dimensional vector space over $K$ and that $f\in \text{Sym}^dV.$ There is an associated symmetric multi-linear form $\tilde f:V^d\to  K$ given by 
$$\tilde f (h_1,\ldots,h_d) = \nabla_{h_1}\ldots\nabla_{h_d} f(0),$$
which satisfies $\tilde f(x,\ldots,x) = d! f(x).$ It is not difficult to see that the strength of $f$ and the partition rank of $\tilde f$ are closely related.

\begin{lemma}[\texorpdfstring{\cite[claim 3.2]{LZ}}{[LZ24a, claim 3.2]}]\label{lem:polar}
    If $K$ has characteristic zero or $>d$ then
    \[
    \str(f) \le \prk(\tilde{f}) \le \binom{d}{\lfloor d/2 \rfloor} \str(f).
    \]
\end{lemma}

We need the following key lemma.

\begin{lemma}[\cite{KLP}, theorems 1.3 and A.1]\label{lem:rk-sing}\mbox{}
    \begin{enumerate}
        \item Suppose that $K$ is \emph{algebraically closed} and that the characteristic of $K$ does not divide $d.$ Then
        \[
        \frac{\Brk(f)}{2} \le \str(f) \le (d-1)\cdot \Brk(f).
        \]
        \item Suppose that $K$ is infinite and that $F:V^d\to K$ is multi-linear. Then 
        \[
        \grk(F)\le \prk(F) \le (2^{d-1}-1)\cdot \ark(F)
        \]
    \end{enumerate}
\end{lemma}
\begin{proof}[Proof of theorem \ref{thm:str-brk} assuming theorem \ref{thm:prk-grk}]
    Let $\overline{\str}(f)$ denote the strength of $f$ in $\overline{K}[x_1,\ldots,x_n],$ where $\overline{K}$ is the algebraic closure of $K.$ Similarly, let $\overline{\prk}(\tilde f)$ denote the partition rank of $\tilde f$ when considered as a multi-linear map over $\overline{K}.$ By lemmas \ref{lem:rk-sing} and \ref{lem:polar} we get 
    \[
    \grk(\tilde f) \le \overline{\prk}(\tilde f) \le \binom{d}{\lfloor d/2 \rfloor} \cdot \overline{\str}(f) \le (d-1) \binom{d}{\lfloor d/2 \rfloor}\cdot \Brk(f).
    \]
    
    By theorem \ref{thm:prk-grk} and lemma \ref{lem:polar} we have 
    \[
     \str(f) \le \prk(\tilde f) \le (2^{d-1}-1)(d-1)^{r+1} \binom{d}{\lfloor d/2 \rfloor}\cdot \Brk(f) 
    \] 
    when $K$ is infinite and 
    \[
    \str(f) \le \prk(\tilde f)  \ll_d \Brk(f)\cdot  \log (\Brk(f)+1) 
    \]
    when $K$ is finite.
\end{proof}
\subsection{Proof of theorem \ref{thm:prk-grk}} Suppose first that $K=\F_q$ is a finite field. For $l\ge 1$ we write $F_l$ for the multi-linear map $F\otimes \F_{q^l}: (V\otimes \F_{q^l})^d \to \F_{q^l}.$ Note the following simple inequality.

\begin{lemma}[\cite{CM}, proof of corollary 1]\label{lem:rk-extensions}
    We have $\prk(F) \le l\cdot \prk(F_l).$
\end{lemma}

The heavy lifting will be done by the following result of  Moshkovitz and Zhu. 
\begin{prop}[\cite{MZ}, theorem 1.6 for $\ep =1$]\label{prop:LR-stable}
    There exists a constant $C= C(d)$ such that the following holds: For any $a\ge 1,$ if $\ark(F_l) \le a$ and  $q^l \ge C(a+1)^{d-2}$ then  $\prk (F_l) \le 2^{d-1} a.$
\end{prop}

\begin{proof}[Proof of theorem \ref{thm:prk-grk} for finite fields]
     Let $r = \grk(F).$ Theorem \ref{thm:ark-grk} implies that \linebreak  $\ark(F_l) \le (d-1)r$ for all $l\ge 1.$  Let
\[
l = \lceil \log C +(d-2) \log\left((d-1)r+1 \right) \rceil,
\]
where $C = C(d)$ is the constant of proposition \ref{prop:LR-stable}. Note that $l \ll_d \log (r+1).$  Our choice of $l$ guarantees that $q^l \ge 2^l \ge C[(d-1)r+1]^{d-2}$ so we may apply proposition \ref{prop:LR-stable} with $a = (d-1)r$ to obtain
\[
\prk(F_l) \le 2^{d-1}(d-1)r.
\]
Applying lemma \ref{lem:rk-extensions} gives
\[
\prk(F) \le l\cdot 2^{d-1}(d-1)r \ll_d r\log(r+1).
\]
\end{proof}
The case of an infinite field is comparatively immediate.
\begin{proof}[Proof of theorem \ref{thm:prk-grk} for infinite fields] By theorem \ref{thm:ark-grk} and lemma \ref{lem:rk-sing}, we have 
\[
\prk(F) \le (2^{d-1}-1)\cdot \ark(F) \le (2^{d-1}-1)(d-1)^r\cdot\grk(F).
\]
\end{proof}
\subsection{Proof of theorem \ref{thm:ark-prk}} Again, $K=\F_q$ is finite and $F:V^d\to K$ is a multi-linear form. We need the following lemma.

\begin{lemma}\label{lem:grk-linear-bound-in-ark}
    We have $\grk(F) \ll_d \ark(F).$
\end{lemma}

\begin{proof}
    The inequality $\ark(F) \ge \grk(F)\cdot (1-\log_q(d-1)) $ was established in \cite[Proposition 4.1]{CM-cubic} (they only state the case $d=3,$ but the proof is identical for arbitrary $d$). We deduce that for $q \ge (d-1)^2$ we have $\grk(F) \le 2\cdot \ark(F).$
    
    To handle small fields, we use the inequality $\ark(F_l) \le l^{d-1}\cdot \ark(F),$ proved in \cite[Proof of corollary 1]{CM}. Choosing $l = \lceil 2\log (d-1) \rceil,$ we have $q^l \ge 2^l\ge (d-1)^2$ so we obtain 
\begin{align*}
    \grk(F) &= \grk(F_l) \le 2\cdot \ark (F_l) \\
    &\le 2l^{d-1} \ark(F) \le 2\cdot (2\log (d-1)+1)^{d-1}\cdot \ark(F).
\end{align*}  
\end{proof}

\begin{proof}[Proof of theorem \ref{thm:ark-prk}]
    By lemma \ref{lem:grk-linear-bound-in-ark} and theorem \ref{thm:prk-grk} we have 
    \[
    \prk(F) \ll_d \grk(F) \log(1+\grk(F)) \ll_d \ark(F) \log(1+\ark(F)).
    \]
\end{proof}

\subsection*{Structure of the proof of theorem \ref{thm:ark-grk}} We begin by applying Weil restriction to reduce to the case of purely transcendental extensions of either $\Q$ or $\F_q.$ By choosing a basis for $V$ we may assume without loss of generality that $F$ has coefficients in $A,$ where $A = \Z[t_1,\ldots,t_s]$ or $A = \F_q[t_1,\ldots,t_s],$ respectively. The quotients of $A$ by maximal ideals are finite fields, and we consider the reductions of $F$ over these finite fields. We prove a characterization of geometric rank via the asymptotic count of singular points for these reductions. 

The chief novel technique in our proof lies in the next step. We employ a scaling inequality for zeroes of systems of multi-linear equations to deduce that many of these singular points have small height and therefore \emph{lift to genuine singular points of} $F$. For finite fields, there is one final step consisting of estimating the size of the fibers of the ``evaluation at zero'' map, which sends singular points with coefficients in $\F_q[t]$ to singular points with coefficients in $\F_q.$

\section{Weil restriction, dimension and Point Counting}
\begin{definition}\label{defn:geom-rk-bound}
    Let $K$ be a field. Let $D_{K,d}$ denote the least positive integer such that for any finite dimensional vector space $V$ over $K$ and any multi-linear map $G:V^{d-1}\to V$, we have 
    \begin{equation}\label{eqn:geom-rk-bound}
    \codim_{V^{d-1}} \overline{Z_G(K)} \le D_{K,d} \cdot \codim_{V^{d-1}} Z_G,
    \end{equation}
    where $Z_G = (G(x) = 0).$
\end{definition}
The point of this section is to show how we can relate the constants $D_{K,d}, D_{L,d}$ when $L$ is an extension field or a residue field of $K$. 
\subsection{Weil Restriction}
Our first result deals with finite separable extensions.

\begin{lemma} [Weil restriction]\label{lemma:separable}
    Let $L/K$ be a finite separable extension of infinite fields. Then $D_{L,d} \le D_{K,d}.$ 
\end{lemma}
\begin{proof}
Let $\ell = [L:K].$ Let $V$ be a finite-dimensional vector space over $L$ and \linebreak $G: V^{d-1}\to V$ a multi-linear form. Denote by $V_K$ the $K$-vector space (of dimension $\ell\cdot \dim V$)  and by $G_K:V_K^{d-1}\to V_K$ the multi-linear form (over $K$) obtained by restricting scalars. By \cite{BLR90}, \S 7.6, Theorem 4, we have:
\begin{equation}\label{eqn:weil}
R_{L/K}(Z_G)\cong Z_{G_K}.
\end{equation}
By \cite{JLMM22}, Corollary 3.8 we have $\ell\cdot \codim_{V^{d-1}}(Z_G) = \codim_{V_K^{d-1}}(Z_{G_K})$. By the definitions of Zariski closure and Weil restriction, we have $Z_{G_K}(K) \subset R_{L/K}(\overline{Z_G(L)}).$
Moreover, $R_{L/K}(\overline{Z_G(L)})$ is closed in $V_K^{d-1}$ (\cite{BLR90}, \S 7.6, Proposition 2), so 
\[
\overline{Z_{G_K}(K)} \subset 
R_{L/K}(\overline{Z_G(L)}).
\]
Taking codimensions, we obtain the inequality
\begin{equation}\label{eqn:codim}
\codim_{V_K^{d-1}} \overline{Z_{G_K}(K)} \geq \ell \cdot \codim_{V^{d-1}} \overline{Z_G(L)}.
\end{equation}
We may now apply \cref{eqn:geom-rk-bound,eqn:weil,eqn:codim} to obtain:
\[
\codim \overline{Z_G(L)} \leq \frac{\codim \overline{Z_{G_K}(K)}}{\ell} \leq \frac{D_{K,d}\codim Z_{G_K}}{\ell} = D_{K,d}\codim Z_G.
\]
\end{proof}
Next, we may reduce the proof of \cref{thm:ark-grk} (1) to the case of purely transcendental extensions of $\Q$ or $\F_q$.
\begin{prop}[Reduction to Pure Transcendental Case]\label{cor:pure-transc}
    Suppose \cref{thm:ark-grk} (1) holds for the fields $K = \Q,\ \F_q(t)$ and their purely transcendental extensions, and suppose \cref{thm:ark-grk} (2) holds. Then \cref{thm:ark-grk} (1) holds.
\end{prop}
\begin{proof}
    We consider three types of infinite field $K$: 
    \begin{itemize}
        \item[(i)] A field of infinite transcendence degree over $\Q$ or $\F_q$
        \item[(ii)] A field of finite transcendence degree over $\Q$ or $\F_q(t)$
        \item[(iii)] An infinite field contained in $\overline{\F_p}$ for some $p$.
    \end{itemize}
    
    In case (i), the result is trivial (and also meaningless). 
    
    For case (ii), let $F$ be a $d$-tensor with coefficients in $K$. Let $K'\subseteq K$ denote the subfield generated by the coefficients of $F$ over the prime field of $K$. By construction, the parameter $r$ satisfies $r(K')\leq r(K).$ Let $V'$ be a $K'$-vector space and let \linebreak $F': (V')^d\to K'$ be such that $F'_K\cong F$.
    
    Every finitely-generated extension of a perfect field admits a separating transcendence basis (see \cite{AAA}, Theorem 3 or \cite{vdW}, Lemma 1), so in particular, $K'$ is isomorphic to a finite separable extension of the rational function field $\Q(t_1,\dots, t_{r(K') - 1})$ or $\F_q(t_1,\dots, t_{r(K')})$. By assumption, \cref{thm:ark-grk} holds for $\Q(t_1,\dots, t_{r(K') - 1})$ and $\F_q(t_1,\dots, t_{r(K')})$, so \cref{lemma:separable} implies that \cref{thm:ark-grk} holds for $K'$. Consequently, we have
    \[
    \ark(F)\leq \ark(F') \leq (d-1)^{r(K')}\grk(F') \leq (d-1)^{r(K)}\grk(F') = (d-1)^{r(K)}\grk(F).
    \]
    We will return to the proof of case (iii) after a brief interlude with necessary results.
\end{proof} 
\begin{definition}[Complexity]\label{defn:complexity}
    Suppose an affine variety $Z\subseteq \A^N$ is cut out (set theoretically) by polynomials $P_1,\dots, P_m$. Then the \textit{complexity} of $Z$ is equal to 
    \[\max(N,m, \deg(P_1), \dots, \deg(P_m)).\]
    \end{definition}
    \begin{lemma}[\cite{T12}, Lemma 1 or \cite{LW}, Lemma 1]\label{lemma:schwarz-zippel}
    Let $Z$ be an affine variety of complexity at most $M$ defined over $\F_q$. Then
    \[
    |Z(\F_q)| \ll_M q^{\dim Z}.
    \]
    \end{lemma}
    \begin{remark*}
        Lang and Weil (\cite{LW}) gave implied constants for the bounds in \cref{lemma:schwarz-zippel} and \cref{prop:lang-weil}) in terms of three parameters: $\dim Z, \deg Z, N$. As in \cite{T12}, we combine these three parameters into one parameter: complexity. While this new parameter is apparently less robust, since we do not care about the implied constants from these two results, \cref{defn:complexity} allows us to simplify the rest of our arguments in this section.
    \end{remark*}   
    As a consequence of \cref{lemma:schwarz-zippel}, we are able to prove case (iii) of \cref{cor:pure-transc}.
    \begin{proof}[End of proof of \cref{cor:pure-transc}]
        Fix a multi-linear form $F$ over $K \subseteq \overline{\F_p}$. The coefficients of $F$ lie in some finite subfield $\F_q$ for some $q = p^e$. Let $V'$ be an $\F_q$-vector space and let $F': (V')^d\to \F_q$ be a multi-linear form such that $F'_K\cong F$.  Write $K = \bigcup_{i>0} \F_{q^{e_i}}$, where $1 = e_1 < e_2 < \dots$ is an increasing sequence of positive integers. Applying \cref{lemma:schwarz-zippel} to $Z = \overline{S_F(K)}$, we obtain 
\[
        \dim Z \geq \lim\sup_{i\to\infty} \log_{q^{e_i}} |Z(\F_{q^{e_i}})|,
\]
    or equivalently
    \begin{equation}\label{eqn:lim-inf-ark}
        \ark(F) \leq \lim\inf_{i\to\infty} \ark (F'_{e_i}).
    \end{equation}
    By \cref{thm:ark-grk} (2), the right-hand side of \cref{eqn:lim-inf-ark} is bounded above by $(d-1)\grk(F)$, proving the claim.
    \end{proof}
    
\subsection{Reduction to Finite Fields and the Lang-Weil Bound}
Kopparty, Moshkovitz, and Zuiddam proved the following result which connects the geometric rank of a $d$-tensor over $\Q$ to the analytic ranks of a family of prime characteristic models.
\begin{prop}[\cite{KMZ20}, Theorem 8.1]
    Let $F: (\Z^n)^d\to \Z$ be a $d$-tensor. Then we have
    \[
\liminf_{p\to\infty} \ark(F\otimes_\Z \F_p) = \grk(F\otimes_\Z\Q).
    \]
\end{prop}
For our purposes, we require an analogous result for a $d$-tensor $F: (\F_q[t]^n)^d\to \F_q[t]$. The remainder of this section will be devoted to the proof of the following proposition.
\begin{prop}\label{prop:liminf-ark}

    Let $A = \F_q[t], K = \F_q(t)$ and let $F:(A^n)^d\to A$ be a $d$-tensor. We have
    \begin{equation}
    \liminf_{\pf\in \max\Spec A, |A/\pf|\to \infty} \ark(F\otimes_A A/\pf) = \grk(F\otimes_A K).
\end{equation}
\end{prop}
        \begin{lemma}[\cite{T12}, Theorem 2 or \cite{LW}, Theorem 1]\label{prop:lang-weil}
            Suppose $Z$ is an affine variety of complexity at most $M$ defined over $\F_q.$ Let $c = c(Z)$ be the number of top-dimensional geometrically irreducible components of $Z.$ Then
            \[
            |Z(\F_q)| = (c + O_M(q^{-1/2}))q^{\dim Z}.
            \]
        \end{lemma}
We are now ready to begin the setup for \cref{prop:liminf-ark}. Let $A = \F_q[t]$ for a finite field $\F_q$. Let $K = \F_q(t) = \text{Frac}(A)$. A variety $X/K$ can be viewed as the generic point of a finite type $A$-scheme $\mathfrak X$. In this subsection, we'll build up results that will allow us to connect properties of $X$ to properties of closed fibers of $\mathfrak X$ over $\Spec(A)$. 

The next two propositions are standard results on constructible properties of schemes. 
\begin{lemma}[\cite{stacks-project}, \href{https://stacks.math.columbia.edu/tag/0559}{Tag 0559}]\label{lemma:geom-fiber} 
    Let $X\to Y$ be a morphism of schemes. Assume
    \begin{itemize}
    \itemsep0em 
        \item $Y$ is irreducible with generic point $\eta$,
        \item $X_\eta$ is geometrically irreducible,
        \item $X\to Y$ is finite type.
    \end{itemize}
    Then there exists a nonempty open $V\subseteq Y$ such that $X_V\to V$ has geometrically irreducible fibers.
    \end{lemma}
        \begin{lemma}[\cite{stacks-project},\href{https://stacks.math.columbia.edu/tag/05F7}{Tag 0559}]\label{lemma:dim-fiber}
        Let $X\to Y$ be a morphism of schemes. Assume that
        \begin{itemize}
        \itemsep0em 
            \item $Y$ is irreducible with generic point $\eta$,
            \item $\dim X_\eta = d$,
            \item $X\to Y$ is finite type. 
        \end{itemize}
        Then there exists a non-empty open $V\subseteq Y$ such that the fibers of $X_V\to V$ are $d$-dimensional.
    \end{lemma}
    \begin{lemma}\label{lemma:lying-over}
        Let $A = \F_q[t]$, and set $K:= \F_q(t)$. Let $L/K$ be a finite extension and let $B$ denote the integral closure of $A$ in $L$. Then there are infinitely many pairs $(\pf\in \Spec A, \qf\in \Spec B)$ such that $\qf\cap A = \pf$ and $A/\pf\to B/\qf$ is an isomorphism. In this construction, each $\pf$ occurs in at most one pair.
    \end{lemma}
    \begin{proof}
        By \cref{lemma:field-emb}, there is an infinite collection $P$ of prime elements $\pi\in A$ such that for each $\pi$, there is a field embedding $\a_\pi: L\to K_\pi,$ where $K_\pi$ is the $\pi$-adic completion of $K.$ Let $\pi\in P$. Since $B$ is integral over $A$ and $A^{\wedge \pi}$ is integrally closed in $K_\pi$, we have $\a_\pi(B)\subseteq A^{\wedge \pi}$.
        
        Consequently, we consider the ring map $\b_\pi: B\xrightarrow{\a_\pi}A^{\wedge\pi}\to A/\pi$. Let $p^e$ denote the inseparable degree of $L/K$. Since $\b_\pi$ extends the Frobenius iterate $F^e: A\to A$, the image of $\b_\pi$ contains $(A/\pi)^{p^e}$. As $A/\pi$ is a finite field, we have that $\b_\pi$ is surjective, hence $B/\ker \b_\pi\cong A/\pi$. As $\pi\in \ker \b_\pi$, the prime ideal $\ker \b_\pi$ lies over $(\pi)$, so the infinite set $\{((\pi), \ker \b_\pi):\pi\in P\}$ satisfies the desired properties.
    \end{proof}
    
    \begin{lemma}\label{lemma:dense-geom-irred-fibers}
        Let $A = \F_q[t]$, and set $K:= \F_q(t)$. Let $\Xf$ be a finite-type $A$-scheme with generic fiber $X$. There are infinitely many primes $\pf\in \Spec A$ such that the irreducible components of the fiber $\Xf_\pf$ are geometrically irreducible and also \linebreak $\dim \Xf_\pf = \dim X$.
    \end{lemma}
    \begin{proof}
                Let $L$ be a finite extension of $K$ such that the irreducible components of $X\times_{\Spec K} \Spec L$ are geometrically irreducible. Let $B$ denote the integral closure of $A$ in $L$ and let $\Yf := \Xf\times_{\Spec A}\Spec B$. Consider the map $\Yf\to \Spec A$. The fiber of $\Yf$ over $K$ is given by
                \[
                \Yf\times_{\Spec A}K \cong \Xf\times_{\Spec A}\Spec (B\otimes_A K) \cong \Xf\times_{\Spec A}\Spec L.
                \]
                Consequently, we denote this fiber by $X_L$.
                
                Applying \cref{lemma:geom-fiber} once for each irreducible component of $X_
                L$, there is a nonempty open set $U\subseteq \Spec A$ such that the fibers of $\Yf$ over each $s\in U$ decompose into geometrically irreducible components. Moreover, by \cref{lemma:dim-fiber}, there exists a nonempty open set $V\subseteq \Spec A$ such that $\dim \Xf_s = \dim X$ for all $s\in V$. 

                Apply \cref{lemma:lying-over} to the extension $K/L$ to obtain an infinite set $Q$ of pairs $\{(\pf,\qf)\}$. Since $A$ is integral and 1-dimensional, the complement of $U\cap V$ in $\Spec A$ is finite, so we may replace $Q$ by an infinite subset $Q'\subseteq Q$ such that $\pf\in U\cap V$ for all $(\pf,\qf)\in Q'$. For each $(\pf,\qf)\in Q'$, factorize $\pf B = \qf^{e_\pf}J_\pf$ for some ideal $J_\pf\subseteq B$ coprime to $\qf$. The reduced fiber of $\Yf$ over $\Spec A/\pf$ is given by
            \begin{align*}
                (\Yf_{\pf})_{\text{red}}&= (\Yf\times_{\Spec A}\Spec (A/\pf))_{\text{red}}\\
                 &\cong (\Xf\times_{\Spec A}(\Spec B/\qf^{e_\pf}\sqcup \Spec B/J_\pf)_{\text{red}} \\&\cong (\Xf_\pf)_{\text{red}}\sqcup (\Xf\times_{\Spec A} \Spec B/J_\pf))_{\text{red}}
            \end{align*}
            As $\pf\in U$, the irreducible components of $\Yf_\pf$ are geometrically irreducible. In particular, as $(\Xf_\pf)_{\text{red}}$ is a connected component of $(\Yf_\pf)_{\text{red}}$, the irreducible components of $(\Xf_\pf)_{\text{red}}$ are geometrically irreducible, hence the same holds for $\Xf_\pf$. Moreover, by assumption that $\pf\in V$, we have $\dim \Xf_\pf = \dim X$.
            \end{proof}
    
    Combining the above lemma with the two point-counting estimates, we have the following.
                \begin{lemma}\label{lemma:reduction}
        Let $A = \F_q[t]$, and set $K:= \F_q(t)$. Let $\Xf$ be an affine, finite-type $A$-scheme with generic fiber $X$. Then we have
            \[
    \limsup_{\pf\in \max\Spec A, |A/\pf|\to \infty} \frac{\log|\Xf_\pf(A/\pf)|}{\log |A/\pf|} = \dim X.
    \]
        \end{lemma}
        \begin{proof}
            By \cref{lemma:dim-fiber}, we have $\dim \Xf_\pf = \dim X$ for all but finitely many primes $\pf$. As the complexity of $\Xf_\pf$ is upper-bounded by the complexity of $X$, by \cref{lemma:schwarz-zippel} we have
            \[
                \limsup_{\pf\in \max\Spec A, |A/\pf|\to \infty} \frac{\log |\Xf_{\pf}(A/\pf)|}{\log|A/\pf|} \leq \dim X.
            \]
            In other direction, \cref{lemma:dense-geom-irred-fibers} implies that for infinitely many primes $\pf\in \Spec A$, $\dim \Xf_{\pf} = \dim X$ and $\Xf_{\pf}$ has a top-dimensional geometrically irreducible component. By \cref{prop:lang-weil} we have the following asymptotic, which completes our proof.
            \[
            \frac{\log |\Xf_{\pf}(|A/\pf|)|}{\log|A/\pf|} \geq  \dim X + \frac{\log (1 + O_X(|A/\pf|^{-1/2}))}{\log|A/\pf|}. 
            \]
        \end{proof}

The proof of \cref{prop:liminf-ark} is an immediate consequence:
\begin{proof}
    Apply \cref{lemma:reduction} to $S_F$ to obtain
    \[
        \limsup_{|A/\pf|\to \infty} \frac{\log|S_{F\otimes_AA/\pf}(A/\pf)|}{\log |A/\pf|}  = \dim S_{F\otimes_A K}.
    \]
    Subtract both sides from $n(d-1)$ to conclude the result.
\end{proof}

\section{Proof of \texorpdfstring{\cref{thm:ark-grk}}{Theorem 1.2} in positive characteristic}

Fix a $d$-tensor $F$ over $K = \F_q(t)$. To bound the analytic rank of $F$ in terms of the geometric rank, one proceeds as follows:
\begin{enumerate}
    \item First, change bases so that $F$ has coefficients in $A = \F_q[t]$. The variety \[\mathfrak S_F = (F(x,\cdot)\equiv 0)\subseteq \A_A^{n(d-1)}\] has generic fiber $S_F$, giving an integral model for $S_F$ over $A$.
    \item Using \cref{prop:liminf-ark}, one shows that many fibers of $\mathfrak S_F$ over $\Spec A/\pf$ have many $A/\pf$-rational points.
    \item Next, using a scaling result, one can promote some of these $A/\pf$-points of $\mathfrak S_F$ to $A$-points  of small height. 
    \item Using Noether normalization, these $A$-points of small height give a lower bound on $\dim S_F$ in terms of $\grk(F)$, which in turn yields the upper bound $\ark(F)\leq (d-1)\grk(F)$.
    \item For finite fields, we extend to $\F_q(t)$, apply the above argument, and then bound the size of fibers of the map evaluating at $t=0.$
\end{enumerate}
By induction on the transcendence degree, one can deduce a bound for $\F_q(t_1,\dots, t_r)$.
We begin with a lemma, inspired by a covering argument in \cite{Sch}, that will be crucial for step (3).

\begin{lemma}[Scaling]\label{lem:scale-charp}
    Let $H,H'$ be finite abelian groups and $G:H^{d-1}\to H'$ a multi-linear map. Suppose $\varphi:H\to L$ is a group homomorphism and define
    \[
    N^y(G) =  |\{x\in H^{d-1}: G(x) = 0, \varphi(x_i) = y_i\}|
    \]
    for any $y = (y_1,\dots, y_{d-1})\in L^{d-1}.$
    Then
    \begin{equation}\label{eq:fiber-size}
        N^y(G) \le N^0(G).
    \end{equation}  
    In particular, for any subgroup $H_0\subset H,$ we have 
    \[
    |\{x\in H^{d-1}: G(x) = 0\}|\le [H:H_0]^{d-1}\cdot \left|\{x\in H_0^{d-1}: G(x) = 0\}\right|.
    \]
\end{lemma}

\begin{proof}
    The statement about subgroups is deduced by taking $\varphi$ to be the projection $\varphi:H\to H/H_0$ and then summing over all $y\in (H/H_0)^{d-1}.$ We prove inequality \eqref{eq:fiber-size} by induction on $d.$
    For the base case $d=2,$ there are two possibilities. If $N^y(G) = 0$, then there is nothing to prove. Otherwise, there exists some $x_0\in H$ with $G(x_0) = 0$ and $\varphi(x_0) = y.$ In this case,
    \[
    N^y(G) =\left| x_0+\{x\in H: G(x) = 0, \varphi(x) = 0\} \right| = N^0(G).
    \]
    
    For general $d>2,$ we have
    \begin{align*}
            N^y(G) &= \sum_{x\in H, \ \varphi(x_{d-1})=y_{d-1}} N^{(y_1,\ldots,y_{d-2})} (G(\cdot, x)) \le \sum_{x\in H, \ \varphi(x_{d-1})=y_{d-1}} N^0 (G(\cdot,x)) \\
         &=  \sum_{x'\in H^{d-2}, \ \varphi(x'_i)=0} N^{y_{d-1}} (G(x',\cdot)) \le \sum_{x'\in H^{d-2}, \ \varphi(x'_i)=0} N^0 (G(x',\cdot)) = N^0(G),
    \end{align*}
   where the first inequality follows from the inductive hypothesis and the second one from the base case $d=2.$
\end{proof}

\subsection{Infinite Fields}

We now define a quantity which refines the notion of analytic rank of a $d$-tensor over a rational function field.
\begin{definition}\label{defn:height-rank-q}
    Let $A = \F_q[t_1,\dots, t_r]$ and let $K = \F_q(t_1,\dots, t_r)$. Let $F: (A^n)^d\to A$ be a multi-linear form. Define $G: (A^n)^{d-1}\to \Hom_A(A^n, A)$ by $G(x) = F(x, \cdot)$. For $R_1,\dots, R_r \in \Z^+$, let
    \[
    N_{R_1,\dots, R_r}(G) = |\{x\in (A^n)^{d-1}: G(x) = 0,\  \deg_{t_i}(x) < R_i\text{ for all }1\leq i\leq r.\}|
    \]
    We then define the following invariant of $F$:
    \[
    \gamma_q(F) = n(d-1) - \limsup_{R_1,\dots, R_r\to \infty}  \frac{\log_q N_{R_1,\dots, R_r}(G)}{R_1\dots R_r}.
    \]
\end{definition}
\begin{prop}\label{prop:F_q-estimate}
    Let $A,K,F$ and $G$ be as in \cref{defn:height-rank-q}. Then $$\gamma_q(F) \leq (d-1)^r\grk(F\otimes_A K).$$
\end{prop}

Before proving this proposition, we will show how it implies \cref{thm:ark-grk} for rational function fields over $\F_q$.
\begin{proof}[Proof of \cref{thm:ark-grk}]
Let $K = \F_q(t_1,\dots, t_r)$ for $r\geq 1$.   A standard application of Noether normalization (\cref{lemma:height-points-p}) gives us $\ark(F)\leq \gamma_q(F)\le (d-1)^r\grk(F)$.
\end{proof}
Now, we return to the proof of \cref{prop:F_q-estimate}.
\begin{proof}
    Write $c = \grk(F\otimes_A K).$ For the base case $r = 1$, we apply \cref{prop:liminf-ark} to $F$. Let $\pf$ be a prime of $A$ such that $\ark_{A/\pf}(F\otimes_A A/\pf) \le c +\ep$. Let $G_\pf: (A^n)^{d-1}\to (A/\pf)^n$ denote the map obtained by reducing the coefficients of $G$ modulo $\pf$. For any $l>0$, let $H_l\subseteq A^n$ denote the subgroup of vectors of polynomials of degree strictly less than $l$.
    
    Let $e = \deg \pf$. By lifting each zero $\bar x\in ((A/\pf)^n)^{d-1}$ of $G\otimes_A A/\pf$ to the  unique representative $x\in (H_e)^{d-1}$ of $G_\pf$, we get
    \begin{equation}\label{eqn:N_e_G_pf}
        N_e(G_\pf) \ge q^{e(n(d-1) - c - \ep)}.
    \end{equation}
    By \cref{lem:scale-charp}, for any $0<\sigma<1,$ 
    \begin{equation}\label{eqn:scale-soln-p}
        N_e(G_\pf) \le [H_e: H_{\sigma e}]^{d-1} N_{\lceil{\sigma e}\rceil}(G_\pf) = q^{(e - \lceil \sigma e \rceil)n(d-1)} N_{\lceil{\sigma e}\rceil}(G_\pf)
    \end{equation}
    Combining \cref{eqn:N_e_G_pf} and \cref{eqn:scale-soln-p} yields
    \begin{equation*}
        N_{\lceil{\sigma e}\rceil}(G_\pf) \ge q^{\lceil\sigma e\rceil(n(d-1) - \frac{c + \ep}{\sigma })}.
    \end{equation*}
    For any vector $x$ with entries in $A,$ let $\norm{x}$ denote the maximum degree of the entries of $x$. There exists a constant $D$ depending on the coefficients of $G$ such that 
\[
    \norm{G(x)} \le D + (d-1)\norm{x}.
\]
    Consequently, when $D + (d-1)\norm{x} < e$, we have $G_\pf(x) = 0 \iff G(x) = 0$. It follows that for $\sigma <  \frac{1}{d-1}$ and sufficiently large $e,$ we have
    \begin{equation}
   N_{\lceil \sigma e \rceil}(G) \ge q^{\lceil\sigma e\rceil(n(d-1) - \frac{c + \ep}{\sigma })}.
    \end{equation}
    Taking $\ep\searrow 0$ and $\sigma\nearrow \frac{1}{d-1}$, we conclude that $\gamma_q(F) \le (d-1)c$ as desired. 
    
    Suppose instead $r\geq 2$. By \cref{lemma:dim-fiber}, there exists a nonempty open subset $U\subseteq \Spec A$ such that $\grk (F\otimes_A\kappa(\pf))= c$ for every $\pf\in U,$ where $\kappa(\pf) = \text{Frac}(A/\pf)$ is the residue field at $\pf.$ The principal prime ideals in $A$ of the form $(\alpha(t_r))$ form a dense subset of $\Spec A$, so $U$ contains infinitely many of these primes. In particular, let $\pf$ be any such prime with $\grk (F\otimes_A \kappa(\pf))= c,$ and write $e = \deg \pf$. As in the base case, let $G_\pf: (A^n)^{d-1} \to (A/\pf)^n$ denote the map obtained by reducing the coefficients of $G$ modulo $\pf$. 

    By the induction hypothesis, we have $\gamma_{q^e}(F\otimes_A A/\pf) \leq (d-1)^{r-1}c.$ We may therefore choose $R_1,\dots, R_{r-1}\in \Z^+$ arbitrarily large with 
    \begin{equation}
         N_{R_1,\dots, R_{r-1}, e}(G_\pf) \geq q^{eR_1\dots R_{r-1}(n(d-1) - (d-1)^{r-1}c -\ep)} .
    \end{equation}
    For $l_1,\ldots,l_r > 0,$ let $H_{l_1,\ldots,l_r} \subset A^n$ be the additive subgroup of vectors with entries whose degree in $t_i$ is less than $l_i$ for all $i.$  Applying \cref{lem:scale-charp} to the function \linebreak  $G_\pf: H_e\to (A/\pf)^n$ with $0<\sigma < \frac{1}{d-1}$ and arguing similarly to the base case, we obtain
    \begin{equation}
        N_{R_1,\dots, R_{r-1},\lceil \sigma e \rceil}(G) \geq q^{ \lceil \sigma e\rceil R_1\dots R_{r-1} \left( n(d-1) - \frac{(d-1)^{r-1}c + \ep}{\sigma} \right)}.
    \end{equation}
    Taking $\ep\searrow 0$ and $\sigma\nearrow\frac{1}{d-1}$ yields $\gamma_q(F) \le (d-1)^r c, $ as desired.
\end{proof}
\subsection{Finite fields}
\begin{lemma}[Fibers of evaluation at zero] \label{lemma:special-ff-scaling}
        Let $F:(\F_q^n)^d \to \F_q$ be a multi-linear map  and let $G:(\F_q[t]^n)^{d-1}\to \F_q[t]^n$ be the multi-linear map given by $G(x)_i = F_{\F_q(t)}(x,e_i).$ Then for any integer $R >1$ we have
    \[
    N_R(G) \le |S_F(\F_q)|\cdot N_{R-1}(G),
    \]
    where $S_F = (F(x,\cdot)\equiv 0).$ 
\end{lemma}

\begin{proof}
    First note that $x(0)\in S_F(\F_q)$ for any $x\in (\F_q[t]^n)^{d-1}$ with $G(x)=0.$ For any $a\in S_F(\F_q),$ let 
    \[
    N_R^a= |\{x\in S_F(\F_q[t]),\ \deg(x) <R, x(0)=a\}|.
    \]
    By lemma \ref{lem:scale-charp}, we have $N_R^a(G) \le N_R^0(G) = N_{R-1}(G).$ The equality follows from the bijection given by $x\mapsto x/t.$
    Summing over $a\in S_F(\F_q)$ yields  
    \[
    N_R(G) = \sum_{a\in S_F(\F_q)} N_R^a(G) \le \sum_{a\in S_F(\F_q)} N_{R-1}(G) = |S_F(\F_q)|\cdot N_{R-1}(G).
    \]
\end{proof}
\begin{proof}[Proof of theorem \ref{thm:ark-grk} for finite fields]
    This follows from the fact that 
    \[\ark(F) \le \gamma_q(F_{\F_q(t)}) \le (d-1)\grk(F).
    \]
    We prove this now. By \cref{lemma:special-ff-scaling}, for any $R \geq 1$, we have
\begin{equation}\label{eqn:product-bound}
    N_R(G) = |N_1(G)| \cdot \prod_{j=1}^{R-1} \frac{N_{j+1}(G)}{N_j(G)} \leq |S_F(\F_q)|^R.
\end{equation}
By \cref{prop:F_q-estimate}, we have $\gamma_q(F_{\F_q(t)}) \leq (d-1) \grk(F)$, so for any $\ep>0$, there exists $R\in \Z^+$ such that
\begin{equation}\label{eqn:point-lowerbound}
     N_R(G) \geq q^{R(n(d-1) - (d-1)\grk(F) - \ep)}.
\end{equation}
Combining \cref{eqn:point-lowerbound} and \cref{eqn:product-bound} we obtain
\[
\ark(F) = \log_q (|S_F(\F_q)|)\leq (d-1)\grk(F) + \ep
\]
for all $\ep>0$, proving the claim.

\end{proof}

\section{Proof of \texorpdfstring{\cref{thm:ark-grk}}{Theorem 1.2} in characteristic zero}
In this section, $F$ denotes a $d$-tensor over $K = \Q(t_1,\dots, t_s)$ for some $s\geq 0$. Using the results of the previous section, we will derive the bound $\ark(F)\leq (d-1)^{s+1}\grk(F)$. 
\begin{enumerate}
    \item First, change bases so that $F$ has coefficients in $A = \Z[t_1,\dots, t_s]$. The variety \[\mathfrak S_F = (F(x,\cdot)\equiv 0)\subseteq \A_A^{n(d-1)}\] has generic fiber $S_F$, giving an integral model for $S_F$ over $A$.
    \item For $s = 0$, the argument is essentially the same as the case of $\F_q[t]$ in the previous section, so we assume $s\geq 1$ for the remainder of this outline.
    \item For all but finitely many primes $p$, the form $F\otimes_A A/p$ has the same geometric rank as $F$. By \cref{prop:F_q-estimate}, the variety $S_{F\otimes_A A/p}$ has many rational points of small height. Similarly to the previous section, many of these points lift to $A$-points of $\mathfrak S_F.$ 
    \item By Noether normalization, the abundance of $A$-points proves the desired upper bound $\ark(F)\leq (d-1)^{s+1}\grk(F)$.
\end{enumerate}

We begin with a semigroup analogue of lemma \ref{lem:scale-charp}.

\begin{lemma}[Scaling]\label{lem:scale-char0}
     Let $B$ be an abelian group and let $G:(\Z^n)^{d-1} \to B$ be a multi-linear map. For a positive integer $R$ set
    \begin{align*}
        N_R(G) &= |\{x\in (\Z^n)^{d-1}: G(x) = 0,\ 0\le x_j< R\  \forall j\in[n(d-1)]\}| \\
        Z_R(G) &= |\{x\in (\Z^n)^{d-1}: G(x) = 0,\ -R< x_j< R\  \forall j\in[n(d-1)]\}|.
    \end{align*}
    Then for any positive integer $L$ we have
    \[
        N_{LR}(G) \le L^{n(d-1)} Z_R(G).
    \]
\end{lemma}

\begin{proof}
    By induction on $d.$
    
    We begin with the base case $d=2$. First note that for any $b \in B$ the number of integer solutions to $G(x) = b$ in $[0,R)^n$ is bounded above by $Z_R(G).$ This follows by subtracting a single solution from each of the others. Therefore
    \begin{align*}
        N_{LR}(G) &= \sum_{0\le x_i<LR} 1_{G(x) = 0} = \sum_{\omega\in\{0,1,\ldots,L-1\}^n} \sum_{0\le x_i<R}  1_{G(x+R\omega) = 0} \\
        &= \sum_{\omega\in\{0,1,\ldots,L-1\}^n} \sum_{0\le x_i<R}  1_{G(x) = -G(R\omega)} \le  \sum_{\omega\in\{0,1,\ldots,L-1\}^n} Z_R(G) = L^n Z_R(G).
    \end{align*}
    
    For general $d>2,$ we have
    \begin{align*}
         N_{LR}(G) &= \sum_{x\in \Z^n, 0\le x_i<LR} N_{LR} (G(x,\cdot)) \le \sum_{x\in \Z^n ,0\le x_i<LR} L^{n(d-2)}Z_R (G(x,\cdot)) \\
         &= L^{n(d-2)} \sum_{y\in \Z^{n(d-2)}, -R < y_i < R}  N_{LR}(G(\cdot,y)) \\
         &\le L^{n(d-1)} \sum_{y\in \Z^{n(d-2)}, -R < y_i < R}  Z_R(G(\cdot,y)) = L^{n(d-1)} Z_R(G),
    \end{align*}
   where the first inequality follows from the inductive hypothesis and the second one from the base case $d=2.$
\end{proof}
\begin{definition}\label{defn:height-rank-0}
    Let $A = \Z[t_1,\dots, t_s]$ and let $K = \Q(t_1,\dots, t_s)$. Let $F: (A^n)^d\to A$ be a $d$-tensor. Define $G: (A^n)^{d-1}\to \Hom_A(A^n, A)$ by $G(x) = F(x, \cdot)$. Let $A_{\geq 0}$ be the additive sub-semigroup of $A$ in which the coefficients of all monomials are non-negative. For $x\in (A^n)^{d-1}$, let $\norm{x}\in \Z^{\geq 0}$ denote the largest absolute value of any integer coefficient appearing in any term of $x$. For $L, R_1,\dots, R_s\in \Z^+$, we define
    \[
    N_{L, R_1,\dots, R_s}(G) = |\{x\in (A_{\geq 0}^n)^{d-1}: G(x) = 0,\ \norm{x} < L,\ \deg_{t_i}(x) < R_i\text{ for all }1\leq i\leq s\}|
    \]
    and similarly define
    \[
        Z_{L, R_1,\dots, R_s}(G) = |\{x\in (A^n)^{d-1}: G(x)=0,\ \norm{x} < L,\ \deg_{t_i}(x) < R_i\text{ for all }1\leq i\leq s\}|.
        \]
    We then define the following invariants of $F$:
    \[
    \gamma_0(F) = n(d-1) - \limsup_{L, R_1,\dots, R_s\to \infty}  \frac{\log_L Z_{L, R_1,\dots, R_s}(G)}{R_1\dots R_s}.
    \]
    Additionally, let $G_L$ denote the map $(A^n)^{d-1}\xrightarrow{G}A^n\to (A/L)^n$. We also define 
    \[
    \delta_0(F) = n(d-1) - \limsup_{L, R_1,\dots, R_s\to \infty}  \frac{\log_L N_{L, R_1,\dots, R_s}(G_L)}{R_1\dots R_s}.
    \]
    Note that this definition makes sense even if $s = 0$, as the empty product equals 1.
\end{definition}
\begin{lemma}\label{lemma:lift-mod-L}
    Let $A,K,F,G$ be as in \cref{defn:height-rank-0}. Then $\gamma_0(F) \leq (d-1)\delta_0(F)$.
\end{lemma}
Assuming this lemma, we are now able to prove \cref{thm:ark-grk} for fields of characteristic zero.
\begin{proof}
    A standard application of Noether normalization (\cref{lemma:height-points-0}) shows that $\ark(F)\leq \gamma_0(F)$.
    By \cref{lemma:lift-mod-L}, it suffices to prove that $\delta_0(F) \leq (d-1)^s\grk(F)$. For $s = 0$, this follows from \cite{KMZ20}, Theorem 8.1. For $s \geq 1$, we note by \cref{lemma:dim-fiber} that for all but finitely many primes $p\in \Z$, we have $\grk(F\otimes_A \kappa(p)) = \grk(A)$. Consequently, the bound $\delta_0(F) \leq (d-1)^s\grk(F)$ follows from \cref{prop:F_q-estimate} for all $s\geq 1$.
\end{proof}
To conclude this section, we now prove \cref{lemma:lift-mod-L}.
\begin{proof}
    Choose $L, R_1,\dots, R_s\in \Z^+$ such that 
    \begin{equation}\label{eqn:NvGL}
         N_{L, R_1,\dots, R_s}(G_L) \geq L^{ R_1\dots R_s(n(d-1) - \delta_0(F) - \ep)}.
    \end{equation}
    Let $R = R_1\dots R_s$. The basis given by monomials yields an isomorphism \linebreak $\Z[t_1,\dots, t_s]_{\deg t_i < R_i}\cong \Z^R$ of abelian groups, which sends $G_L$ to a multi-linear map  $\tilde G_L:(\Z^{Rn})^{d-1} \to (A/L)^n.$  Under this identification, $\norm{x}$ denotes the largest absolute value of any entry of the corresponding vector in $(\Z^{Rn})^{d-1}.$ For $0<\sigma<1$, \cref{eqn:NvGL} together with \cref{lem:scale-char0} gives us
        \[
            L^{R(n(d-1) - \delta_0(F) - \ep)} \leq N_L(\tilde G_L) \leq \ceil{L^{1-\sigma}}^{Rn(d-1)} Z_{\ceil{L^\sigma}}(\tilde G_L) \ll_{d,n,R} L^{R(1-\sigma)n(d-1)} Z_{\ceil{L^\sigma}}(G_L).
        \]
        In particular, we have 
        \begin{equation}\label{eqn:Z_L-estimate}
        Z_{\ceil{L^\sigma}}(G_L) \gg_{d,n,R} L^{R\sigma \left(n(d-1) - \frac{\delta_0(F) + \ep}{\sigma} \right)} \gg_{d,n,R} \ceil{L^\sigma}^{Rn(d-1) - \frac{\delta_0(F) + \ep_{\vec v}}{\sigma}}.
        \end{equation}
        
        For any $x \in (\Z[t_1,\dots, t_s]_{\deg t_i < R_i}^n)^{d-1}$ satisfying $\|x\| < \ceil{L^\sigma}$, we have \linebreak $\norm{G(x)} \ll_F L^{\sigma(d-1)}$. Therefore, for $L$ sufficiently large, $\sigma < \frac{1}{d-1}$ and $\|x\| < \ceil{L^\sigma}$,
        \[
        G_L(x) = 0 \mod L \iff G(x) =  0.
        \]
        Consequently, \cref{eqn:Z_L-estimate} yields
        \[
        \gamma_0(F) \leq \frac{\delta_0(F) + \ep}{\sigma}.
        \]
        Taking the limit $\ep\searrow 0$ and $\sigma\nearrow \frac{1}{d-1}$ completes the proof.
\end{proof}

\appendix
\section{}
The results contained within this appendix are standard and similar to existing results. Nevertheless, we detail their proofs here for the sake of self-containedness and completeness.
\subsection{A Consequence of Noether Normalization}
\begin{lemma}[Height point growth in positive characteristic]\label{lemma:height-points-p}

    Let $A = \F_q[t_1,\dots, t_r]$ and $K = \F_q(t_1,\dots, t_r)$. Suppose $X\subset \A^n$ is an affine variety defined over $K.$ We define a height function on points of $A^n$ by $H(x) =(\deg_{t_1}x,\ldots,\deg_{t_r}x).$ 
    For $R_1,\dots, R_r\in \N$,  set
    \[
    X(R_1,\dots, R_r) = \{x\in X(A): H(x)_i<R_i \text{ for all }1\leq i\leq r\}.
    \] 
    Then
    \[
        \dim(X) \geq \limsup_{R_1,\dots, R_r\to \infty}\frac{\log_q|X(R_1,\dots, R_r)|}{R_1\dots R_r}.
    \]
\end{lemma}

\begin{proof}
    The key is that we have
    \[
        |\A^b(R_1,\dots, R_r)| = |\A^1(R_1,\dots, R_r)|^b = q^{bR_1\dots R_r}.
    \]
    By Noether normalization, there exists a finite map $\pi: X\to \A^b$ for $b = \dim X$. Since $A$ is infinite, this map can be taken linear with coefficients in $A$. Consequently, there exists $C>0$ such that 
    \[
    \pi(X(R_1,\dots, R_s))\subseteq \A^b(R_1+C,\dots, R_r+C).
    \]
    Taking cardinalities of both sides, we obtain
    \[
    |X(R_1,\dots, R_r)|\leq \deg(\pi)q^{b(R_1+C)\dots(R_r+C)}.
    \]
    Taking $\log_q$, dividing by $R_1\ldots R_r$ and then taking the limit $R_1,\dots, R_r\to \infty$ completes the proof. 
\end{proof}

\begin{lemma}[Height point growth in characteristic zero]\label{lemma:height-points-0}
    Let $A = \Z[t_1,\dots, t_s]$ and $K = \Q(t_1,\ldots,t_s)$. Suppose $X\subset \A^n$ is an affine variety defined over $K.$ We define a height function on points of $A^n$ by  $H(x) = (L, \deg_{t_1}x,\dots, \deg_{t_s}x)$, where $L$ is the maximum of the absolute values of the coefficients in $x.$
    For $L,R_1,\dots, R_s\in \N$, set 
    \[
    X(L,R_1,\dots, R_s) = \{x\in X(A): H(x)_0<L, H(x)_i < R_i\text{ for all }1\leq i\leq s\}.
    \]
    Then
    \[
        \dim (X) \ge \limsup_{L, R_1,\dots, R_s\to \infty}\frac{\log_L(X(L, R_1,\dots, R_s))}{R_1\dots R_s}\geq a.
    \]
\end{lemma}
\begin{proof}
    The key is that 
    \[
        |\A^b(L,R_1,\dots, R_s)| = |\A^1(L,R_1,\dots, R_s)|^b = (2L-1)^{bR_1\dots R_s} \leq (2L)^{bR_1\dots R_s}.
    \]
    By Noether normalization, there exists a finite map $\pi: X\to \A^b$ for $b = \dim X$. Since $A$ is infinite, this map can be taken linear with coefficients in $A$.  Consequently, there exists $C>0$ such that 
    \[
    \pi(X(L,R_1,\dots, R_s))\subseteq \A^b(CL, R_1+C,\dots, R_s+C).
    \]
    Taking cardinalities of both sides, we obtain
    \[
    |X(L, R_1,\dots, R_s)|\leq \deg(\pi)(2CL)^{b(R_1+C)\dots(R_s+C)}.
    \]
    Taking $\log_L$, dividing by $R_1\dots R_s$ and taking the limit  $L,R_1,\dots, R_s\to \infty$ completes the proof. 
\end{proof}

\subsection{Embeddings into Local Fields}
In \cite{C76}, Cassels constructs embeddings from a finitely generated extension of $\Q$ into $\Q_p$ for infinitely many primes $p$. For our purposes, we require a weaker result for finite extensions of $\F_q(t)$. These arguments are minor modifications of Cassels's own.
\begin{lemma}\label{lemma:nonvanishing}
Let $A = \F_q[t]$. Let $G(X) \in A[X]$ be a non-constant function. Then $G$ has a solution mod $\pi$ for infinitely many prime elements $\pi \in A$.
\end{lemma}

\begin{proof}
Cassels gives an elementary proof of this result in \cite{C76A} for $A = \mathbb{Z}$. 
The proof for $\F_q[t]$ is identical, with only one claim requiring special justification in this case.
\begin{claim}
    Let $P$ be a finite set of primes in $\F_q[t]$ and suppose $G(0)\neq 0$. Then there exists $c$ divisible by all $\pi\in P$ such that $G(G(0)c)$ is not a unit multiple of $G(0)$.
\end{claim}
To prove this claim, it suffices to take $c = t^n\prod_{\pi\in P}\pi$ for $n\gg 0$.
\end{proof}

\begin{lemma}\label{lemma:field-emb}
Let $A = \F_q[t],\ K = \F_q(t)$ and let $L/K$ be a finite extension. Let $e\geq 0$ such that $L^{p^e}\subseteq K^{\text{sep}}$. Then there are infinitely many primes $\pi\in A$ which admit a field embedding $\a_\pi: L \to K_\pi$ extending the Frobenius iterate $F^e: A\to A$.
\end{lemma}
\begin{proof}
Start with the case that $L/K$ is separable and write $L = K(\beta)$ using the primitive element theorem. Suppose $\beta$ satisfies minimal equation $H(z) = H_s z^s + \dots + H_0$ over $K$. Clear denominators so that $H_i \in A$. Let $P'$ denote the set of prime elements $\pi \in A$ such that $H(z) = 0$ has a solution mod $\pi$. Remove from $P'$ all primes which divide $(H_s \Delta(H))$. By \cref{lemma:nonvanishing}, the resulting set $P$ is still infinite. Let $\pi \in P$. By the fact that 
\[\Delta H \not\equiv 0 \mod \pi,\]
it follows that $H$ has a \textit{simple} root mod $\pi$. By Hensel's lemma, it follows that $H$ has a root $\eta \in A^{\wedge \pi}$. We then define $\a_\pi(\beta) = \eta$, which gives a well-defined embedding $\a_\pi: L\to K_\pi$ extending the identity map $A\to A$. In the case that $L/K$ is inseparable, we apply the above argument to $L^{p^e}/K$ for some $e$ such that $L^{p^e}\subseteq K^{\text{sep}}$ and precompose with $F^e: L\to L^{p^e}$.
\end{proof}

\printbibliography

\end{document}

%% file: lettercoms.tex
\newcommand{\A}{\mathbf{A}}

\newcommand{\F}{\mathbf{F}}
\newcommand{\N}{\mathbf{N}}
\newcommand{\Q}{\mathbf{Q}}

\newcommand{\Z}{\mathbf{Z}}

\renewcommand{\P}{\mathbb P}

\newcommand{\pf}{\mathfrak p}
\newcommand{\qf}{\mathfrak q}

\renewcommand{\a}{\alpha}
\renewcommand{\b}{\beta}

\renewcommand{\emptyset}{\varnothing}

%% file: declarations.tex
\DeclareMathOperator{\Hom}{Hom}
\DeclareMathOperator{\Spec}{Spec}

\DeclareMathOperator{\Xf}{\mathfrak X}
\DeclareMathOperator{\Yf}{\mathfrak Y}

\DeclareMathOperator{\codim}{codim}

\newcommand{\ceil}[1]{\lceil #1\rceil}

\usepackage{array}   
\newcolumntype{C}{>{$}c<{$}} 

\newcommand\norm[1]{\lVert#1\rVert}

\renewcommand{\ker}{\textup{ker }}
\renewcommand{\char}{\textup{char }}

\renewcommand{\bar}{\overline}